\documentclass[10pt]{amsart}
\usepackage{amsmath,amsthm,amsfonts,amssymb,amscd,flafter,pinlabel}
\usepackage{mathtools}
\usepackage[mathscr]{eucal}
\usepackage{graphics}
 \usepackage[all]{xy}
\usepackage{caption, subcaption}
\usepackage[usenames,dvipsnames]{color}
\usepackage{graphicx}

\usepackage{pgf}
\usepackage{tikz}
\usetikzlibrary{arrows,automata}
\usepackage[latin1]{inputenc}

\usepackage[colorlinks=true, urlcolor=NavyBlue, linkcolor=NavyBlue, citecolor=NavyBlue, pdftitle={Symplectic surfaces and bridge position}, pdfauthor={Peter Lambert-Cole}, pdfsubject={symplectic}, pdfkeywords={4-manifolds}]{hyperref}

\newtheorem{theorem}{Theorem}[section]
\newtheorem{corollary}[theorem]{Corollary}
\newtheorem{conjecture}[theorem]{Conjecture}

\newtheorem{proposition}[theorem]{Proposition}

\newtheorem{lemma}[theorem]{Lemma}

\theoremstyle{definition}
\newtheorem{definition}[theorem]{Definition}

\theoremstyle{definition}

\newtheorem{problem}[theorem]{Problem}

\newcommand\M{\mathcal M}

\newcommand{\cM}{\mathcal{M}}

\newcommand{\CC}{\mathbb{C}}
\newcommand{\RR}{\mathbb{R}}
\newcommand{\del}{\partial}

\newcommand{\ZZ}{\mathbb{Z}}

\newcommand{\cD}{\mathcal{D}}

\newcommand{\cT}{\mathcal{T}}

\newcommand{\cA}{\mathcal{A}}
\newcommand{\cC}{\mathcal{C}}
\newcommand{\cK}{\mathcal{K}}

\newcommand{\cB}{\mathcal{B}}

\newcommand{\DD}{\mathbb{D}}

\newcommand{\CP}{\mathbb{CP}}

\begin{document}

\title[{Symplectic surfaces and bridge position}]{Symplectic surfaces and bridge position}

\author[P. Lambert-Cole]{Peter Lambert-Cole}
\address{School of Mathematics \\ Georgia Institute of Technology}
\email{plc@math.gatech.edu}
\urladdr{\href{http://people.math.gatech.edu/~plambertcole3}{http://people.math.gatech.edu/\~{}plambertcole3}}

\maketitle


\begin{abstract}

We give a new characterization of symplectic surfaces in $\CP^2$ via bridge trisections.  Specifically, a minimal genus surface in $\CP^2$ is smoothly isotopic to a symplectic surface if and only if it is smoothly isotopic to a surface in transverse bridge position.  We discuss several potential applications, including the classification of unit 2-knots, establishing the triviality of Gluck twists, the symplectic isotopy problem, Auroux's proof that every symplectic 4-manifold is a branched cover over $\CP^2$, and the existence of Weinstein trisections.  The proof exploits a well-known connection between symplectic surfaces and quasipositive factorizations of the full twist in the braid group.
\end{abstract}

\section{Introduction}

Trisections of 4-manifolds are generalizations of Heegaard splittings of 3-manifold.  The projective plane $\CP^2$ admits a trisection compatible with its complex and toric geometry.  This is a decomposition $\CP^2 = Z_1 \cup Z_2 \cup Z_3$ where each $Z_{\lambda}$ is a bidisk $\DD \times \DD$, each double intersection $H_{\lambda} = Z_{\lambda -1} \cap Z_{\lambda}$ is a solid torus $S^1 \times \DD$ foliated by holomorphic disks, and the triple intersection $Z_1 \cap Z_2 \cap Z_3$ is a Lagrangian torus.  

Just as links can be isotoped into bridge position with respect to a Heegaard splitting, knotted surfaces can be put into bridge position with respect to a trisection \cite{MZ-Bridge,MZ-GBT}.  A surface $\cK \subset \CP^2$ is in {\it bridge position} if $\cK$ intersects each solid torus $H_{\lambda}$ along a boundary-parallel tangle and each sector $Z_{\lambda}$ along boundary-parallel disks.  We furthermore say that $\cK$ is in {\it transverse bridge position} if it is in bridge position and each tangle $\tau_{\lambda}$ is positively transverse to the foliation on $H_{\lambda}$ by holomorphic disks. In \cite{LC-Thom}, the author used transverse bridge position and other notions of transversality for surfaces in $\CP^2$ to give a new proof of the Thom conjecture.  

Interestingly, transverse bridge position characterizes symplectic surfaces among all minimal genus surfaces.

\begin{theorem}
\label{thrm:trans-equals-symplectic}
Let $(\CP^2,\cK)$ be an embedded surface of degree $d > 0$ and satisfying $g(\cK) = \frac{1}{2}(d-1)(d-2)$.  Then $\cK$ is smoothly isotopic to a symplectic surface if and only if it is smoothly isotopic into transverse bridge position.
\end{theorem}

This characterization of symplectic surfaces, along with the main results in \cite{LC-Thom,Square}, suggest that there is a very natural connection between symplectic topology and trisections of 4-manifolds.

The proof relies on a different, well-known description of symplectic surfaces in $\CP^2$ as quasipositive braided surfaces (also known as {\it Hurwitz curves}).  Symplectic surfaces of degree $d$ in $\CP^2$ can be encoded by braid factorizations of the full twist $\Delta_d^2$ in $B_d$ and a pair of symplectic surfaces are isotopic if and only if their braid factorizations are equivalent under so-called Hurwitz moves.  Following Rudolph \cite{Rudolph}, we can turn this algebraic data into a banded link presentation of the surface.  It is then straightforward to turn a banded link presentation into a bridge presentation of the surface.

\subsection{Symplectic Isotopy problem}

Every nonsingular complex curve in $\CP^2$ is symplectic with respect to the Fubini-Study Kahler form $\omega_{FS}$.  Determing whether the converse is true is a very interesting open problem. 

\begin{problem}[Symplectic Isotopy Problem]
Suppose that an oriented, connected, nonsingular surface $\cK \subset \CP^2$ is symplectic with respect to $\omega_{FS}$.  Is $\cK$ isotopic through symplectic surfaces to a complex curve?
\end{problem}

Successive work of Gromov, Sikorav, Shevchishin and Siebert-Tian has established that the problem has a positive answer if the degree of $\cK$ is at most 17 \cite{Gromov,Sikorav, Shevchishin,Siebert-Tian}.  In particular, Gromov showed that every symplectic sphere of degree 1 or 2 is isotopic through symplectic surfaces to a holomorphic curve.  The basic approach is to choose an $\omega_{FS}$-tame almost-complex structure $J_0$ making a sympectic surface $J$-holomorphic, then study the deformation problem through a 1-parameter family $J_t$ connecting $J_0$ to the standard complex structure.  Progress on this question has for the most part stalled, although there has been some recent interest from the perspective of contact geometry and symplectic fillings \cite{Starkston}. 

Building on work of Fintushel and Stern \cite{FS-1,FS-2}, Finashin used rim surgery to construct infinite families of surfaces of degree $d \geq 5$ that were topologically equivalent but not smoothly isotopic to algebraic plane curves \cite{Finashin,Finashin2}.  Subsequently, Kim constructed infinite families of exotic algebraic curves for $d \geq 3$ \cite{Kim}.  Combining these results with Theorem \ref{thrm:trans-equals-symplectic}, we obtain the following corollary.

\begin{corollary}
The surfaces constructed by Finashin and Kim for $d \leq 17$ cannot be isotoped into transverse bridge position.
\end{corollary}

\subsection{Unit 2-knots}

A {\it 2-knot} is a pair $(X,\cK)$ where $\cK$ is an embedded 2-sphere.  A {\it unit 2-knot} in $\CP^2$ is a 2-knot $(\CP^2,\cK)$ where $[\cK]$ represents the (positive) generator of $H_2(\CP^2;\ZZ)$.  The {\it standard unit 2-knot} in $\CP^2$ is the 2-knot $(\CP^2,\CP^1)$, where $\CP^1$ is defined in homogeneous coordinates as the set $\left\{ [x:y:z] : x = 0 \right\}$.  A unit 2-knot $(\CP^2,\cK)$ is {\it standard} if it is smoothly isotopic to the standard unit 2-knot.

Combining Theorem \ref{thrm:trans-equals-symplectic} with Gromov's result that every symplectic unit 2-knot is standard \cite{Gromov}, we have the following result.

\begin{theorem}
\label{thrm:symp-unit-2-knot}
Let $(\CP^2,\cK)$ be a unit 2-knot.  If $\cK$ is in transverse bridge position, then it is isotopic to the standard unit 2-knot $(\CP^2,\CP^1)$.
\end{theorem}

A long-standing conjecture is that every unit 2-knot is standard (see Problem 4.23 of Kirby's problem list).

\begin{conjecture}[see \cite{Melvin,Kirby-problem}]
\label{conj:unit-2-standard}
Every unit 2-knot $(\CP^2,\cK)$ is smoothly isotopic to $(\CP^2,\CP^1)$.
\end{conjecture}

We can therefore rephrase Conjecture \ref{conj:unit-2-standard} as follows.

\begin{conjecture}
Every unit 2-knot in $\CP^2$ can be isotoped into transverse bridge position.
\end{conjecture} 

\subsection{Gluck twists}

Let $(X,\cK)$ be a 2-knot with trivial normal bundle.  Recall that the {\it Gluck twist} on a 2-knot $(X,\cK)$ with trivial normal bundle is the manifold obtained by surgery on $\cK$, cutting and regluing by the unique nontrivial, orientation-preserving, diffeomorphism of $S^1 \times S^2 = \del \nu(\cK)$ \cite{Gluck}.  A Gluck twist on a 2-knot $(S^4,\cK)$ always results in a homotopy 4-sphere, which is then homeomorphic to $S^4$ by Freedman's Theorem.  Gluck twists are one potential source of counterexamples to the Smooth 4-dimensional Poincar\'e conjecture.  Melvin described a connection between Gluck twists on $S^4$ and unit 2-knots in $\CP^2$ \cite{Melvin}.  Let $(S^4,\cK)$ be a 2-knot and take the connected sum $(\CP^2,F_{\cK}) = (S^4,\cK) \# (\CP^2, \CP^1)$.  Then $F_{\cK}$ is a smooth $+1$-sphere and we can blow down to obtain a homotopy 4-sphere.  Melvin showed that the result of blowing down along $F_{\cK}$ is the Gluck twist $S^4_{\cK}$.

Thus, another immediate application is of Theorem \ref{thrm:symp-unit-2-knot} is a criterion to show that a Gluck twist is trivial.

\begin{theorem}
Let $(S^4,\cK)$ be a 2-knot.  If $(\CP^2,\cK \# \CP^1)$ can be isotoped into transverse bridge position, then the Gluck twist on $\cK$ is diffeomorphic to $S^4$.
\end{theorem}

Gromov's result has been well-known for several decades, but on its own is not a particularly useful criterion.  Transverse bridge position, however, appears to be a much more useful criterion for establishing that a given unit 2-knot is standard and that a specific Gluck twist is trivial.

Connected sums can be easily described in terms of shadow diagrams of bridge trisected surfaces.  See \cite{MZ-Bridge} for diagrams of 2-knots in $S^4$ and and \cite{MZ-GBT,LM-complex} for shadow diagrams for $(\CP^2,\CP^1)$.

\subsection{Symplectic Branched Covers}

An important structural result in symplectic topology, due to Auroux, is that every closed symplectic 4-manifold admits a symplectic branched covering map to $\CP^2$.

\begin{theorem}[Auroux \cite{Auroux-Branched}]
\label{thrm:Auroux}
Let $X$ be a closed symplectic 4-manifold.  There exists a branched covering $f: X \rightarrow \CP^2$ such that $f(B)$, the image of the branch locus in $\CP^2$, is a nodal, cuspidal symplectic surface.
\end{theorem}

The trisection perspective is well-suited for understanding coverings of 4-manifolds, both regular and branched over a surface in bridge position. The author and Meier constructued many trisections of Kahler surfaces by taking cyclic branched covers over complex curves in rational complex surfaces \cite{LM-complex}.  Cahn and Kjuchukova also constructed iregular dihedral covers of $S^4$ using trisections \cite{CK}. Implicit in the former constructions is that complex curves can be isotoped into transverse bridge position.  

The proof of Theorem \ref{thrm:trans-equals-symplectic} also applies to the branch loci in Auroux's construction.  Specifically, by generalizing the notion of bridge position to nodal, cuspidal surfaces in $\CP^2$ we can obtain the following result.

\begin{theorem}
\label{thrm:branch-locus}
Let $f: X \rightarrow \CP^2$ be a symplectic branched cover and $f(B)$ the branch locus in $\CP^2$.  Then $f(B)$ can be isotoped into singular transverse bridge position.
\end{theorem}

\subsection{Weinstein trisections}

The notion of a Weinstein trisection was introduced in \cite{LM-complex}.  Let $(M,\omega)$ be a closed symplectic 4-manifold.  A {\it Weinstein trisection} of $M$ consists of a trisection $M = Z_1 \cup Z_2 \cup Z_3$ of $X$ and a Weinstein structure $(Z_{\lambda}, \omega|_{Z_{\lambda}},X_{\lambda}, \phi_{\lambda})$ on each sector such that the Liouville vector field $X_{\lambda}$ is outward-pointing along $Y_{\lambda} = \del Z_{\lambda}$.  

Note that in a trisection, each sector $Z_{\lambda}$ is a 4-dimensional 1-handlebody, which abstractly admits a (flexible) Weinstein structure.  Akbulut and Matveyev showed that every 4-manifold can be covered by two Stein domains \cite{AkMat}.  However, the two Stein structures have no geometric compatibility on the overlap.  For a Weinstein trisection, we require that the symplectic geometry match up on the overlap.

It was shown in \cite{LM-complex} that $\CP^2$ admits a Weinstein trisection (in fact a {\it Stein trisection}).  Auroux's result can be used to pull this back to any closed symplectic 4-manifold $X$.  Given a branched covering $f: X \rightarrow \CP^2$, the branch locus can be put into transverse bridge position (Theorem \ref{thrm:branch-locus}).  After possibly a perturbation, the preimage of the Weinstein trisection on $\CP^2$ is a Weinstein trisection.

\begin{theorem}[\cite{Square}]
Every symplectic 4-manifold admits a Weinstein trisection.
\end{theorem}

\subsection{Transverse bridge position in $(X,\omega)$}

Transverse bridge position in $\CP^2$ is defined with respect to taut foliations on the double intersections of a specific trisection of $\CP^2$.  Given a Weinstein trisection of $(X,\omega)$, it is possible to find such taut foliations and extend the notion of transverse bridge position to general symplectic 4-manifolds.  Suppose that $(M,\omega)$ admits a Weinstein trisection.  The $\alpha_{\lambda} = \iota_{X_{\lambda}} \omega$ restricts to a contact form on $Y_{\lambda}$.  Over $H_{\lambda} = Z_{\lambda - 1} \cap Z_{\lambda}$, there are two contact forms $\alpha_{\lambda-1}$ and $\alpha_{\lambda}$, one inducing a negative contact structure and the other a positive contact structure.  Their difference, $\beta_{\lambda} = \alpha_{\lambda} - \alpha_{\lambda-1}$ is a closed, nonvanishing 1-form dominated by the symplectic form $\omega$.  Therefore the kernel field of $\beta_{\lambda}$ is a taut foliation.

By a result of Meier and Zupan \cite{MZ-GBT}, every embedded surface $\cK$ in $M$ can be isotoped into bridge position.  When $M$ admits a Weinstein trisection, it is possible to use these taut foliations to extend the notion of transverse bridge position to surfaces in $\M$.  A straightforward adaptation of the proof of Proposition \ref{prop:trans-implies-symplectic} implies that every surface in transverse bridge position can be made symplectic. The converse seems likely as well.  We therefore

\begin{conjecture}
Let $\cK \subset (X,\omega)$ be a connected, oriented surface with $\chi(\cK) = \langle c_1(\omega),[\cK] \rangle - [\cK]^2$.  Then $\cK$ is isotopic to a symplectic surface if and only if it is isotopic into transverse bridge position with respect to some Weinstein trisection of $(X,\omega)$.
\end{conjecture}

\subsection{Acknowledgements}

I would like to thank Maggie Miller and Laura Starkston for comments and encouragement.

\section{Bridge trisections and $\CP^2$}

In this section, we describe the standard trisection of $\CP^2$ and review some of the results and terminology in \cite{LC-Thom}.

\subsection{Bridge trisections}

\begin{definition}

A $(g;k_1,k_2,k_3)$-{\it trisection} of a smooth, closed oriented 4-manifold $X$ is a decomposition $X = Z_1 \cup Z_2 \cup Z_3$ such that
\begin{enumerate}
\item each $Z_{\lambda}$ is diffeomorphic to $\natural_{k_{\lambda}} S^1 \times D^3$,
\item each double intersection $H_{\lambda} = Z_{\lambda - 1} \cap Z_{\lambda}$ is a 3-dimensional 1-handlebody of genus $g$, and
\item the triple intersection $\Sigma = Z_1 \cap Z_2 \cap Z_3$ is a closed surface of genus $g$.
\end{enumerate}

\end{definition}

Let $\{\tau_i\}$ be a collection of properly embedded arcs in a handlebody $H$.  An arc collection is {\it trivial} if they can be simultaneously isotoped to lie in $\del H$.  A {\it bridge splitting} of a link $L$ is the 3-manifold $Y$ is a decomposition $(Y,L) = (H_1, \tau_1) \cup_{\Sigma} (H_2,\tau_2)$ where $H_1,H_2$ are handlebodies and the arc collections $\tau_1,\tau_2$ are trivial.  Finally, a collection $\cD = \{\cD_i\}$ of properly embedded disks in a 1-handlebody $X$ are {\it trivial} if they can be simultaneously isotoped to lie in $\del X$.

\begin{definition}
\label{def:bridge-trisection}
A $(b;,c_1,c_2,c_3)$ {\it bridge trisection} of a knotted surface $(X,\cK)$ is a decomposition $(X,\cK) = (X_1,\cD_1) \cup (X_2,\cD_2) \cup (X_3,\cD_3)$ such that
\begin{enumerate}
\item $X = Z_1 \cup Z_2 \cup Z_3$ is a trisection of $X$,
\item each $\cD_{\lambda}$ is a collection of $c_{\lambda}$ trivial disks in $Z_{\lambda}$, and
\item each tangle $\tau_{\lambda} = \cD_{\lambda - 1} \cap \cD_{\lambda}$, for $i \neq j$, is trivial.
\end{enumerate}
If $(X,\cK)$ admits a bridge trisection, we say that $\cK$ is in {\it bridge position}.  The parameter $b$ is called the {\it bridge index} of $\cK$.
\end{definition}

Let $K_{\lambda} \subset Y_{\lambda}$ be the boundary of the trivial disk system $\cD_{\lambda}$.  Since $\cD_{\lambda}$ is trivial, the link $K_{\lambda}$ is the unlink with $c_{\lambda}$ components. The {\it spine} of a surface $\cK$ in bridge position is the union $\tau_{1} \cup \tau_{2} \cup \tau_{3}$.  The spine uniquely determines the generalized bridge trisection of $\cK$ \cite[Corollary 2.4]{MZ-GBT}.  If $\cK$ admits a $(b;c_1,c_2,c_3)$ bridge trisection, then $\chi(\cK) = c_1 + c_2 + c_3 - b$.

\subsection{Bridge trisections of singular surfaces}

Let $F$ be a surface in a smooth 4-manifold $M$ that is smooth, except possibly at a finite number of points.  We say that $F$ has an  {\it $A_n$-singularity} at $x$ if, in a neighborhood of $x$, we can choose complex coordinates $(x,y)$ on $M$ such that $F$ is the zero locus of the polynomial $x^2 = y^{n+1}$.  Equivalently, there exists a small $B^4$-neighborhood of $x$ such that $F$ intersects along the cone of the torus link $T(2,n+1)$.  We will also refer to an $A_1$-singularity as a {\it node} and an $A_2$-singularity as a {\it cusp}.  The singularity is {\it positive} if the orientation on $F$ agrees with the orientation induced by the compex chart; otherwise it is {\it negative}.

Let $(B^4, C_n)$ denote the pair where $C_K$ is the cone on the link $T(2,n+1)$.  If $n = 0,1$, we assume that the cone is chosen smooth.  Let $(B^4,F)$ be a surface tangle with $A$-singularities.  We say that $(B^4,F)$ is {\it trivial} if it decomposes as a boundary-connected sum
\[ (B^4,F) = (B^4, C_{n_1}) \natural  \cdots \natural (B^4, C_{n_j})\]
for some positive integers $n_1,\dots,n_j$.  Here, the boundary connected sum is done away from the links on the boundary and results in a split link on the boundary.  With this definition of trivial surface tangle, we can immediately extend Definition \ref{def:bridge-trisection} to define bridge trisections of surfaces with $A$-singularities.

\subsection{Trisection of $\CP^2$}

The toric geometry of $\CP^2$ yields a trisection $\cT$ as follows.  Define the moment map $\mu: \CP^2 \rightarrow \RR^2$ by the formula
\[\mu([z_1:z_2:z_3]) \coloneqq \left( \frac{3 |z_1^2|}{|z_1|^2 + |z_2|^2 + |z_3|^2}, \frac{3 |z_2|^2}{|z_1|^2 + |z_2|^2 + |z_3|^2} \right).\]
The image of $\mu$ is the convex hull of the points $\{ (0,0),(3,0),(0,3) \}$.  The barycentric subdivision of the simplex $\mu(\CP^2)$ lifts to a trisection decomposition of $\CP^2$.  Define subsets
\begin{align*}
Z_{\lambda} &\coloneqq \left\{ [z_1:z_2:z_3] : |z_{\lambda}|,|z_{\lambda+1}| \leq |z_{\lambda-1}| \right\} & 
H_{\lambda} &\coloneqq \left\{ [z_1:z_2:z_3] : |z_{\lambda}| \leq  |z_{\lambda-1}| = |z_{\lambda+1}| \right\} .
\end{align*}
In the affine chart on $\CP^2$ obtained by setting $z_3 = 1$, the handlebody $Z_1$ is exactly the polydisk 
\[ \Delta = \DD \times \DD = \{(z_1,z_2) : |z_1|,|z_2| \leq 1 \}.\]
Its boundary is the union of two solid tori $H_{1} = S^1 \times \DD$ and $H_{2} = \DD \times S^1$.  The triple intersection $Z_1 \cap Z_2 \cap Z_3$ is the torus $\Sigma \coloneqq \{[e^{i \theta_1}: e^{i \theta_2}: 1]: \theta_1,\theta_2 \in [0,2\pi] \}$.

Each $Z_{\lambda} \cong \DD \times \DD$ is Stein, although the boundary $Y_{\lambda} = \del Z_{\lambda}$ is not smooth.   This polydisk can be approximated by a holomorphically convex 4-ball.  Specifically, consider the function
\[f_{\lambda,N}(z_{\lambda},z_{\lambda+1}) \coloneqq \epsilon(|z_{\lambda}|^2 + |z_{\lambda+1}|^2) + |z_{\lambda}|^{2N} + |z_{\lambda+1}|^{2N}\]
for some $\epsilon > 0$ and $N \gg 0$. It is strictly plurisubharmonic with a single, nondegenerate critical point of index 0 at the origin.  Thus, the compact sublevel set $\widehat{X}_{\lambda,N} = f_{\lambda,N}^{-1}((-\infty,1])$ is a Stein domain.  The field of complex tangencies along the boundary $\widehat{Y}_{\lambda,N} = \del \widehat{Y}_{\lambda,N}$ is a contact structure.  By choosing $N$ sufficiently large, we can assume $\widehat{Y}_{\lambda,N}$ is $C^0$-close to $Y_{\lambda}$ and $C^{\infty}$-close to $Y_{\lambda}$ outside some fixed neighborhood of the central surface of the trisection.

\subsection{Orientation conventions}

The complex structure on $\CP^2$ determines an orientation on $\CP^2$. Each sector $Z_{\lambda}$ has an orientation and the boundary $Y_{\lambda} = \del Z_{\lambda}$ inherits an orientation.  In order to preserve cyclic symmetry, we fix orientations on each handlebody $H_{\lambda}$ as follows: the handlebody $H_{\lambda}$ is oriented as a subset of $Y_{\lambda}$.  Thus, as oriented manifolds, we have a decomposition
\[Y_{\lambda} = H_{\lambda} \cup_{\Sigma} - H_{\lambda+1}\]
The central surface $\Sigma$ inherits an orientation as the boundary of $H_{\lambda}$ and this orientation is independent of $\lambda$.  

If $\cK$ is oriented, then we can orient a bridge trisection of $\cK$ as follows.  The disk tangle $\cD_{\lambda}$ inherits an orientation from $\cK$ and this induces an orientation on its boundary.  We orient the tangle $\tau_{\lambda}$ to agree with this orientation.  Thus, as oriented pairs, we have a decomposition
\[(Y_{\lambda}, K_{\lambda}) = (H_{\lambda},\tau_{\lambda}) \cup_{\Sigma} -(H_{\lambda+1}, \tau^r_{\lambda+1})\]
With these conventions, the induced orientation on  the points of $\del \tau_{\lambda}$ is independent of $\lambda$ and moreover agrees with their induced orienation as the transverse intersection $\Sigma \pitchfork \cK$.

\subsection{Trisection diagrams in $\CP^2$}

In homogeneous coordinates, the handlebody $H_{\lambda}$ can equivalently be defined as 
\[ H_{\lambda} \coloneqq \left\{ [z_1:z_2:z_3] : |z_{\lambda}| \leq 1, |z_{\lambda+1}| = 1, z_{\lambda - 1} = 1 \right\}\]
Using standard polar coordinates
\[z_{\lambda} = r_{\lambda} e^{i \theta_{\lambda}} \qquad z_{\lambda+1} = r_{\lambda+1} e^{i \theta_{\lambda+1}}\]
we have coordinates $(\theta_{\lambda+1}, r_{\lambda},\theta_{\lambda})$ on $H_{\lambda} = S^1 \times \DD$.  The solid torus $H_{\lambda}$ is foliated by holomorphic disks.  The plane field tangent to this foliation is the kernel of the 1-form $d \theta_{\lambda+1}$.

The three handlebodies $H_1,H_2,H_3$ meet along the central surface $\Sigma = T^2$.  The boundary of a holomorphic disk in $H_{\lambda}$ is a simple closed curve $\alpha_{\lambda}$ that inherits an orientation from the complex structure on $\CP^2$.  Algebraically, the triple of curves $\{\alpha_1,\alpha_2,\alpha_3\}$ satisfies
\[\alpha_{\lambda+1} = - \alpha_{\lambda} - \alpha_{\lambda+1} \qquad \langle \alpha_{\lambda},\alpha_{\lambda+1}\rangle = 1\]
for every $\lambda = 1,2,3$ and with the indices considered mod 3.  We will also use the notation $\alpha = \alpha_1; \beta = \alpha_2; \gamma = \alpha_3$.  In diagrams, we will view $T^2$ as a square with opposite sides identified and choose the triple $\{\alpha,\beta,\gamma\}$ so that $\alpha$ is horizontal and oriented to the right; $\beta$ is vertical and oriented upward; and $\gamma$ is a $(-1,-1)$ curve oriented down and to the left.

Let $B_{\lambda}$ be a core circle of $H_{\lambda}$.  Define the projection map $\pi_{\lambda}: H_{\lambda} \smallsetminus B_{\lambda} \longrightarrow \Sigma$ in coordinates by
\[\pi_{\lambda}(\theta_{\lambda+1},r_{\lambda},\theta_{\lambda}) \coloneqq (\theta_{\lambda}, \theta_{\lambda+1})\]
Let $(\CP^2,\cK)$ be an immersed surface in general position.  Set $\cA = \pi_1(\tau_{1})$, $\cB = \pi_2(\tau_{2})$ and $\cC = \pi_3(\tau_{3})$.  After a perturbation of $\cK$, we can assume that the projections $\cA,\cB,\cC$ are mutually transverse and self-transverse, with intersections away from the bridge points.  In diagrams, our color conventions are that $\cA$ consists of red arcs, $\cB$ consists of blue arcs, and $\cC$ consists of green arcs.  

\subsection{Transverse bridge position}  Recall that each handlebody $H_{\lambda}$ is foliated by holomorphic disks.

\begin{definition}
A knotted surface $(\CP^2,\cK)$ is {in \it transverse bridge position} if $\cK$ is in bridge position with respect to the standard trisection and each tangle $\tau_{\lambda} = \cK \pitchfork H_{\lambda}$ is positively transverse to the foliation of $H_{\lambda}$ by holomorphic disks.
\end{definition}

Diagrammatically, with our conventions, this can be understood as follows. The projection$\cA$, oriented from $(-)$ bridge points to $(+)$ bridge points, must move monotonically upward; the projection $\cB$ must move monotonically to the left; and the projection $\cC$ must move monotonically down and to the right (with respect to a foliation of the torus by lines of slope 1).  

The connection between transverse bridge position and contact geometry is given by the following results.

\begin{proposition}[\cite{LC-Thom}]
Let $(\CP^2,\cK)$ be a surface of degree $d > 0$ in transverse bridge position.  Then
\begin{enumerate}
\item For $N \gg 0$, the link $\widehat{K}_{\lambda} = \cK \cap \widehat{Y}_{\lambda,N}$ is a transverse unlink.
\item The total self-linking number satisfies
\[sl(\widehat{K}_1) + sl(\widehat{K}_2) + sl(\widehat{K}_3) = d^2 - 3d - b\]
where $b$ is the bridge index of $\cK$.
\end{enumerate}
\end{proposition}

\section{Transverse to symplectic}

\begin{lemma}
\label{lemma:symp-spine}
Suppose that $(\CP^2,\cK)$ is in transverse bridge position.  There exists an isotopy of $\cK$, through surfaces in transverse bridge position, so that $\cK$ is symplectic in a neighborhood of the spine $H_1 \cup H_2 \cup H_3$.
\end{lemma}

\begin{proof}
First, we can assume that $\cK$ agrees with a projective line near the bridge points.  Specifically, let $[x:y:1] \in \cK \cap \Sigma$ be a bridge point.  Then $\cK$ can be isotoped through surfaces in transverse bridge position to locally agree with the line $\{\frac{\zeta}{x} z_0 + \frac{\zeta^2}{y} z_1 + z_2 = 0\}$, where $\zeta$ is a primitive $3^{\text{rd}}$-root of unity.  The choice of $\zeta = e^{\pm \frac{2 \pi i}{3}}$ depends on the orientation of the intersection point.

Next, since each tangle $\tau_{\lambda}$ is positively transverse to the foliation of $H_{\lambda}$ by holomorphic disks, the complex line $\CC \langle \tau'_{\lambda} \rangle$ is transverse to $TH_{\lambda}$ in $T\CP^2$ at each point of the tangle $\tau_{\lambda}$.  After an isotopy of $\cK$, fixed along $\tau_{\lambda}$, we assume the tangent planes of $\cK$ along $\tau_{\lambda}$ agree with the complex lines $\CC \langle \tau'_{\lambda} \rangle$ at every point.  This implies that $\cK$ is symplectic is a sufficiently small neighborhood of its spine. 
\end{proof}

\begin{proposition}
\label{prop:max-symp-disk}
Let $(W^4,\omega,\rho) \hookrightarrow (X^4,\omega)$ be a Weinstein domain embedded in a symplectic manifold.  Let $F$ be an oriented surface in $X$ such that 
\begin{enumerate}
\item $F$ intersects $W$ in a boundary-parallel disk tangle, 
\item $F$ is symplectic in some neighborhood $U$ of $\del W$, and
\item $F$ intersects $\del W$ along a transverse unlink with maximal self-linking number.  
\end{enumerate}
Then there is some neighborhood $V \subset \subset U$ of $\del W$ and a smooth isotopy of $F$ supported in $W \smallsetminus V$ such that $F$ is symplectic in $W$.
\end{proposition}

\begin{proof}
The boundary $Y = \del W$ is a hypersurface of contact type, since the Liouville vector field is positively transverse to $Y$.  Let $\lambda = \iota_{\rho}\omega$ be the 1-form $\omega$-dual to $\rho$ and let $\alpha = \lambda|_Y$ be its restriction to $Y$.  By flowing along $\rho$ in the negative direction, we obtain a symplectic embedding of the symplectization $((-\infty,0] \times Y, d(e^t \alpha))$ into $W$. By assumption, we can choose some small $\epsilon > 0$ such that for $-\epsilon < t \leq 0$, the intersection $L_t \coloneqq F \cap \{-t\} \times Y$ is a transverse, $c$-component unlink.  Projecting away the $\RR$-factor, we can view $L_t$ as a transverse isotopy.  

For each component of $L_0$, choose some point in $S^3$ and Darboux chart coordinates $x,y,z$ such that the contact structure on $\xi$ is given by
\[ \xi = \text{ker}(dz + x dy - y dx) \]
In these coordinates, the Liouville form is $\alpha = \alpha_i = f _i(dz + x dy - y dx)$ for some positive, nonvanishing function $f_i$.  Choose some $\delta_i > 0$ such that $d \alpha_i$ is an area form on the unit disk of radius $\delta_i$ in the plane $\{z = 0\}$.  Let $U_i$ be the boundary of this disk; it is a transverse unknot of self-linking number -1.

Unlinks are transversely simple.  Thus, we can extent $L_t$ to a transverse isotopy $L: [-1,0] \times \left( \coprod S^1 \right) \rightarrow Y$, such that $L_{-1}$ is the union of the unlinks $U_1 \cup \dots \cup U_c$.  Let $\phi: [-1,0] \rightarrow (-\RR,0]$ denote a smooth, nondecreasing function.  Consider the map
\[\widetilde{L}: [-1,0] \times  \left( \coprod S^1 \right) \rightarrow (\phi(t), L_t) \subset \RR \times Y\]
Pulling back $\omega = d(e^t \alpha)$ by $\widetilde{L}$ we obtain
\[\widetilde{L}^*(\omega) = e^{\phi(t)} \phi'(t) dt \wedge L^*(\alpha) + e^{\phi(t)} L^*(d \alpha)\]
Since $L_t$ is a transverse for all $t$, when $\phi$ is increasing we have that $e^{\phi(t)} \phi'(t) dt \wedge L^*(\alpha) > 0$.  Moreover, if $\phi' \gg 0$, then $\widetilde{L}^*(\omega)$ is a positive area form and the image of $\widetilde{L}$ is a symplectic surface.  Near $t = 0$, we know by assumption that $F$ is symplectic.  In addition, it follows from the construction that $L^*(d \alpha)$ is a positive area form for $t$ near $-1$.  It is now clear that we can choose some function $\phi$ such that $\phi(t) = t$ near $t = 0$ and $\phi'(t) = 0$ near $t = -1$ and such that $\widetilde{L}^*(\omega)$ is everywhere a positive area form.  In other words, its image is symplectic.  The disk bounded by each $U_i$ is symplectic, thus we can cap to get a collection of symplectic disks.  After a perturbation, we can assume that the projection $\RR \times Y \rightarrow \RR$, restricted to each disk, is Morse with a single critical point of index 0.  This implies that each disk is isotopic into $\{0\} \times Y$.  Thus, up to ambient isotopy, we can replacing the interior of $F$ with these symplectic disks.
\end{proof}

\begin{proposition}
\label{prop:trans-implies-symplectic}
Suppose that $(\CP^2,\cK)$ be in transverse bridge position and $g(\cK) = \frac{1}{2}(d-1)(d-2)$.  Then $\cK$ is isotopic through surfaces in transverse bridge position to a symplectic surface.
\end{proposition}

\begin{proof}
By Lemma \ref{lemma:symp-spine}, we can assume isotope $\cK$ to be symplectic near its spine.  Furthermore, by \cite[Proposition 3.9]{LC-Thom}, we can assume $\cK$ intersects each $\widehat{Y}_{\lambda,N}$ along a transverse unlink $\widehat{K}_{\lambda}$.

Now, suppose $\cK$ is in $(b;c_1,c_2,c_3)$-bridge position.  This implies that $\chi(\cK) = c_1 + c_2 + c_3 - b$.  The surface $\cK$ has minimal genus, so the Euler characteristic also satisfies $\chi(\cK) = 3d - d^2$.  By \cite[Proposition 3.12]{LC-Thom}, we have that
\[sl(\widehat{K}_1) + sl(\widehat{K}_2) + sl(\widehat{K}_3) = d^2 - 3d - b.\]
The Bennequin bound implies that $sl(\widehat{K}_{\lambda}) \leq - c_{\lambda}$.  Combining these relations, we see that $sl(\widehat{K}_{\lambda}) = c_{\lambda}$ for $\lambda = 1,2,3$.  Thus, each $\widehat{K}_{\lambda}$ is a maximal unlink.  By Proposition \ref{prop:max-symp-disk}, the link $\widehat{K}_{\lambda}$ bounds a trivial symplectic disk system in $\widehat{X}_{\lambda,N}$.  Since each of these disks is boundary-parallel, the surface obtained by capping of the symplectic spine of $\cK$ with these symplectic disks is smoothly isotopic to the original surface $\cK$.
\end{proof}

\section{Symplectic to Transverse}

The main result of this section is the second half of Theorem \ref{thrm:trans-equals-symplectic}.

\begin{proposition}
\label{prop:symp-implies-trans}
Suppose that $\cK$ is symplectic with respect to $\omega_{FS}$.  Then $\cK$ can be isotoped into tranverse bridge position.
\end{proposition}

We isotope a symplectic surface into transverse bridge position in two steps.  First, we use a well-known correspondence between symplectic surfaces of degree $d$ in $\CP^2$ and quasipositive factorizations of the full twist $\Delta^2$ in the braid group $B_d$.  Following Rudolph \cite{Rudolph}, we can turn this into a banded link presentation of the surface.  Finally, we take the banded link presentation and obtain a bridge presentation.

\subsection{Braided surfaces}

Fix a point $\infty \in \CP^2$.  The pencil of complex lines through $\infty$ determines a holomorphic projection map $\pi: \CP^2 \smallsetminus \{\infty\} \rightarrow \CP^1$.  

\begin{definition}
Let $F$ be a smooth surface in $\CP^2$ with only $A$-singularities.  Then $F$ is {\it braided} (with respect to the pencil determined by $\pi$) if 
\begin{enumerate}
\item $F$ is disjoint from the point $\infty$,
\item $F$ is everywhere transverse to the fibers of $\pi$, except at finitely many points where it has a nondegenerate tangency,
\item no $A_k$-singularity of $F$ is tangent to the fibers of $\pi$, and
\item the $A$-singularities and tangencies of $F$ are distinct and are mapped by $\pi$ to distinct points in $\CP^1$.
\end{enumerate}
\end{definition}

If $F$ is smooth and braided, then $\pi: F \rightarrow \CP^1$ is a simple branched covering map.  If $F$ is braided and of degree $d > 0$, then a generic projective line through $\infty$ intersects $F$ transversely in exactly $d$ points.  Moreover, if $F$ is braided and has only positive tangencies to the fibers, then it is isotopic through braided surfaces to a symplectic surface.

\begin{proposition}
\label{prop:braid-equals-symp}
Let $F$ be a nonsingular symplectic surface in $(\CP^2,\omega_{FS})$.  Then $F$ is smoothly isotopic to a braided surface.
\end{proposition}

\begin{proof}
Fix an almost-complex structure $J$ that is compatible with $\omega_{FS}$ and such that $F$ is $J$-holomorphic.  Gromov \cite{Gromov} proved that for any two points $x,y$ in $\CP^2$, there is a unique $J$-holomorphic line through $x$ and $y$.  Thus, for any point $\infty \in \CP^2$, the $J$-holomorphic lines through $\infty$ determine a Lefschetz pencil and a $J$-holomorphic projection $\pi: \CP^2 \smallsetminus \{\infty\} \rightarrow \CP^1$.  Fix a point $\infty \notin F$.  The projection $F \rightarrow \CP^1$ is honestly holomorphic.  Generically, this will have a finite number of Morse critical points, locally modeled on the map $z \mapsto z^2$ with distinct critical values in $\CP^1$.  These critical points correspond exactly to positive tangencies of $F$ with the fibers of the pencil.  At regular values, holomorphicity and positivity of intersections imply that $F$ intersects the fibers positively transversely.
\end{proof}

A refinement of Theorem \ref{thrm:Auroux}, due to Auroux and Katzarkov, states we can assume the image of the branch locus is a (singular) braided surface.

\begin{theorem}[\cite{Auroux-Katzarkov}]
Let $(X,\omega)$ be a closed symplectic 4-manifold such that $\frac{1}{2\pi}[\omega]$ is integral.  Then there exists a singular branched covering map $f: X \rightarrow Y$, ramified along a smooth surface $R \subset X$, such that $f(R)$ is a braided, singular surface with at worst cusp singularities (i.e. $A_1$ and $A_2$ singularities).
\end{theorem}

\subsection{Encoding braided surfaces algebraically}

A braided surface $F$ of degree $d$ in $\CP^2$ can be encoded algebraically as follows.

Fix a generic line $L$ through $\infty \in \CP^2$.  The complement of $L$ can be identified with $\CC^2 = (x,y)$ such that $\pi$ is the projection $\pi(x,y) = x$.  Let $U$ be any compact disk in the base of $\pi$.  The intersection of $F$ with $\pi^{-1}(\del U)$ is a $d$-strand braid $\beta_U$ in the (noncompact) solid torus $S^1 \times \CC$.  If $U$ does not contain any critical values of $\pi|_F$, then this braid $\beta_U$ is trivial.  If $U$ contains the image of a single tangency, then $\beta_U$ is conjugate to $\sigma_1^{\epsilon}$, where $\epsilon$ is the sign of the tangency.  If $U$ contains the image of a single $A_k$-singularity, then $\beta_U$ is conjugate to $\sigma_1^{\epsilon k}$, where $\epsilon$ is the sign of the singularity.  If $U$ contains all of the critical values of $\pi|_F$, then $\beta_U$ is the full twist $\Delta^2_d$ in the braid group $B_d$.  In particular, the surface $F$ determines a braid word
\[ \left( g_1 \sigma_1^{\epsilon_1 k_1} g_1^{-1} \right) \cdots \left( g_n \sigma_1^{\epsilon_n k_n} g_n^{-1} \right) = \Delta_d^2.\]
The exponential sum of a braid word is preserved under the braid relations, so we must have that
\[ \epsilon_1 k_1 + \dots + e_n k_n = d(d-1).\]
Rudolph described how to view $\widehat{F}  = F \cap (\CP^2 \smallsetminus L)$ as a ribbon surface in $B^4$ bounded by $T(d,d) \subset \del B^4 = S^3$ \cite{Rudolph}.  Recall that $B_d$ can be identified with $\cM(D^2,\{x_1,\dots,x_d\})$, the mapping class group disk with $d$ marked points.  Let $\{\gamma_i\}$ be a fixed collection of arcs in $D^2$, with disjoint interiors and with $\del \gamma_i = \{x_i,x_{i+1}\}$.  We get an identification of $\phi: B_d \cong \cM(D^2,\{x_1,\dots,x_d\})$ by sending the Artin generator $\sigma_i$ to the half-twist on $\gamma_i$.  Moreover, any conjugate $g \sigma_i g^{-1}$ is sent to the half-twist on the arc $\phi(g)(\gamma_i)$.  Suppose that $F$ determines the factorization 
\[\Delta_d^2 =  \left( g_1 \sigma_1^{\epsilon_1 k_1} g_1^{-1} \right) \cdots \left( g_n \sigma_1^{\epsilon_n k_n} g_n^{-1} \right).\]
Consider the trivial braid $U = \{x_1,\dots,x_d\} \times S^1 \subset D^2 \times S^1$.  Fix $n$ points $0 < t_1 < t_2 < \dots < t_n < 2\pi$.  For $i = 1,\dots,n$, let $q_i$ denote the arc $\phi(g_i)(\gamma_1) \subset D^2 \times \{t_i\}$.  The trivial braid $U$ bounds $d$ Seifert disks in $S^3$; we can isotope their interiors to lie in the interior of $B^4$.  Now, thicken each arc $q_i$ to a band with a single, positive half-twist.  Up to isotopy, the surface $\widehat{F}$ is the union of the Seifert disks with these bands.

\begin{lemma}
\label{lemma:symp-fact}
Every embedded symplectic surface $F \subset \CP^2$ of degree $d > 0$ determines a factorization of the form
\[\Delta_d^2 =  \left( g_1 \sigma_1 g_1^{-1} \right) \cdots \left( g_n \sigma_1 g_n^{-1} \right)\]
where $n = d^2 - d$.
\end{lemma}

\begin{proof}
Each term in the factorization corresponds to a positive tangency, which implies that corresponding half-twist is positive.  The adjunction equality implies that $\chi(F) = 3d - d^2$ and $F$ has a handle decomposition with $d$ 0-handles, $n$ 2-handles, and $d$ 2-handles.  This implies that $n = d^2 - d$.
\end{proof}

\subsection{Transverse bridge position}

We can now give a bridge presentation of a braided surface $F$ in terms of a torus diagram.  By Proposition \ref{prop:braid-equals-symp} and Lemma \ref{lemma:symp-fact}, we obtain a factorization of the full twist 
\[\Delta_d^2 = \Pi_{i = 1}^n \left( g_i \sigma_1 g_i^{-1} \right)\]
The torus diagram for $F$ is the vertical concatenation of the standard piece depicted in Figure \ref{fig:pos-conj}, one for each term $g_i \sigma_1 g_i^{-1}$ in the factorization.  In particular, we stack these pieces in reverse-order: start at the top with the piece associated to the $n^{\text{th}}$-term $g_n \sigma_1 g_n^{-1}$ and finishing at the bottom with the piece associated to the $1^{\text{st}}$-term $g_1 \sigma_1 g_1^{-1}$.  This will determine a surface, although it will not necessarily be in bridge position as the arcs of $\tau_{\alpha}$ may not be boundary-parallel.  To fix this, we can stabilize as in Figure \ref{fig:mini-stabilization-crossing} to remove any crossings in the diagram for the braids $g_i,g_i^{-1}$.  It is then immediately clear that the arcs of $\tau_{\alpha}$ are boundary-parallel.  Provided that the arc to be stabilized is directed up and to the right, as it is in Figure \ref{fig:mini-stabilization-crossing}, we can preserve transverse bridge position.

\begin{figure}[h!]
\centering
\labellist
	\large\hair 2pt
	\pinlabel $g_i$ at 150 310
	\pinlabel $g^{-1}_i$ at 150 72
	\pinlabel $+$ at 172 232
	\pinlabel $+$ at 250 230
	\pinlabel $-$ at 160 152
	\pinlabel $-$ at 240 150
	\pinlabel $\dots$ at 105 230
	\pinlabel $\dots$ at 105 150
\endlabellist
\includegraphics[width=.3\textwidth]{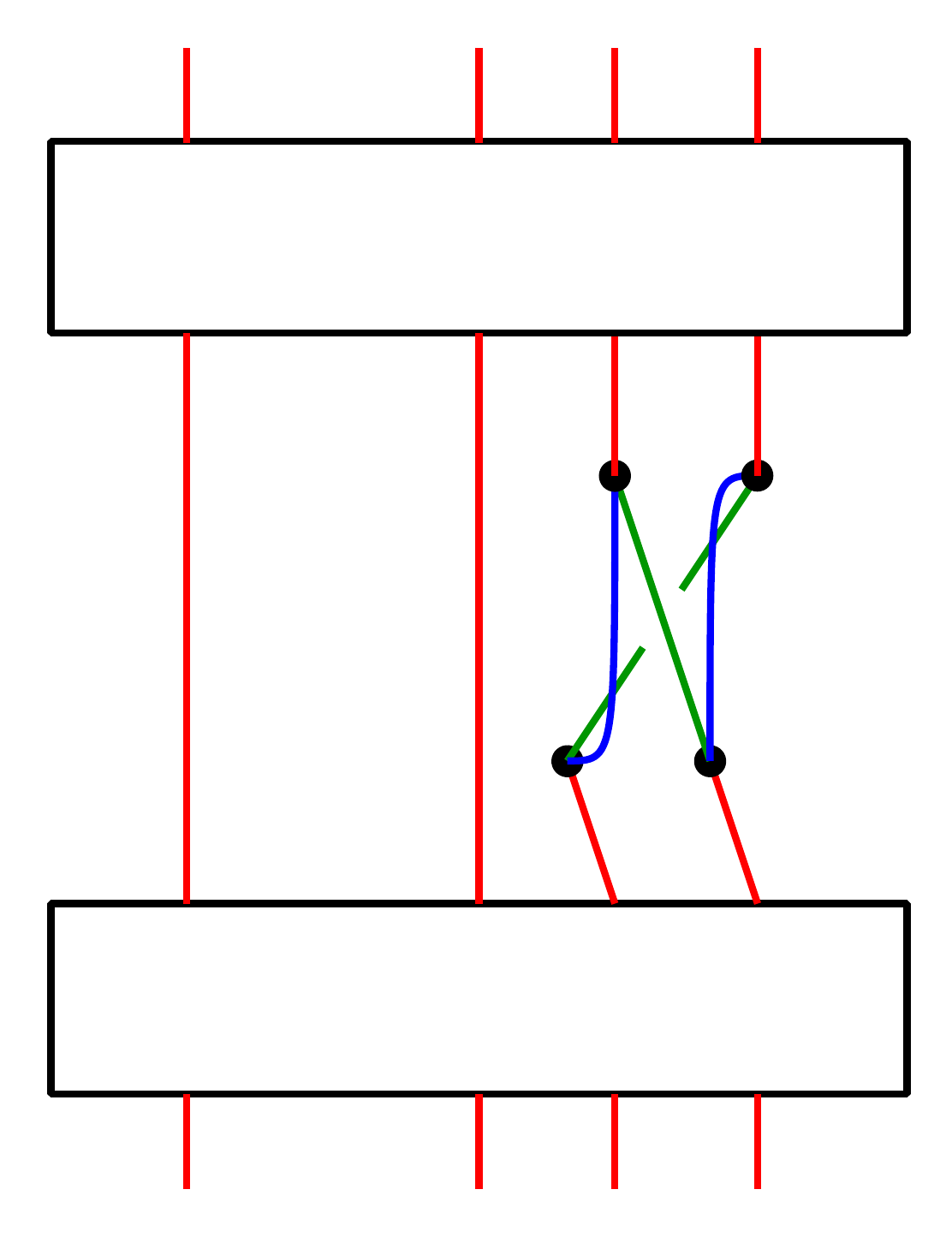}
\caption{The tile of the torus diagram corresponding to the factor $g_i \sigma_1 g_i^{-1}$. Orient each arc away from the $+$ bridge points.  Then the arcs of $\tau_{\alpha}$ (in red) move monotonically upward; the arcs of $\tau_{\beta}$ (in blue) move monotonically to the left; and the arcs of $\tau_{\gamma}$ (in green) move monotonically down and to the right.}
\label{fig:pos-conj}
\end{figure}
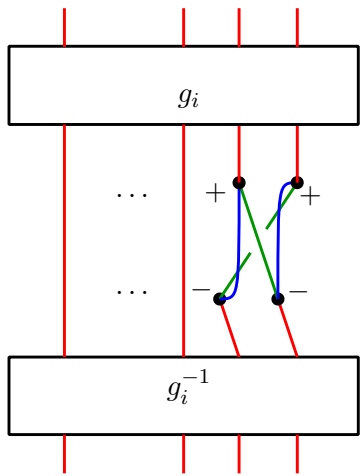

\begin{proposition}
\label{prop:smooth-torus-diagram}
This is a torus diagram for a $(2d(d-1);d,d(d-1);d)$-bridge trisection of $F$.
\end{proposition}

\begin{proof}
First, we check that the diagram determines a surface in bridge position.   To do this, we check that each link obtained by taking the pairwise union of tangles is the unlink.  

First, the link $L_1 = \tau_{\alpha} \cup - \tau_{\beta}$ can be isotoped to be the closure of the trivial $d$-component braid in the solid torus $H_{\alpha}$.  Thus, it is the $d$-component unlink and bounds a collection of boundary-parallel disks.  Note that we view the braid as oriented in the {\it positive} vertical direction.  Second, the link $L_2 = \tau_{\beta} \cup - \tau_{\gamma}$ consists of $d(d-1)$ split unknots, one for each band in the quasipositive factorization, plus several extra split unknots resulting from the mini stabilizations used to remove crossings.  Again, it clearly is an unlink and therefore bounds a collection of boundary-parallel disks.

Finally, consider the link $L_3 = \tau_{\gamma} \cup -\tau_{\alpha}$.  It is also isotopic to a braid closure in $H_{\alpha}$.  Reading this braid word from top to bottom, we obtain
\[g_n \sigma_1^{-1} g_n^{-1} \cdots g_1 \sigma_1^{-1} g_1^{-1} = \Delta_d^{-2}.\]
Thus, $L_3$ appears to be the closure of the negative full twist.  However, the solid torus $H_{\alpha}$ sits in the 3-manifold $Y_3$ and it is the $\gamma$ curve the bounds a disk in the exterior, not the (horizontal) $\beta$ curve.  The $\beta$ curve is in fact a (1,1) curve on the Heegaard torus $\Sigma \subset Y_3$.  In particular, the link obtained by taking $d$ surface-framed pushoffs of the $\beta$ curve is actually $T(d,d) \subset Y_3$.  The link $L_3$ is obtained by adding a negative full twist to this link, thus it is isotopic to the $d$-component unlink.  Consequently, all three links are the unlink and to build a surface, we can cap off with boundary-parallel disks to obtain a surface in bridge position.

Now, we check that the resulting surface is in fact the same as that which produced the braid factorization.  We can view the disks bounded by $L_1$ as $d$ Seifert disks for the unlink.  Moreover, the disk bounded by the unknot component of $L_2$ in the $i^{\text{th}}$ local model can be viewed as a band attached to the Seifert disks.  Finally, we cap off with $d$ Seifert disks bounded by $L_3$.  Abstractly, we can imagine pushing $L_3$ into $H_{\alpha}$ and viewing it as a link in $Y_1$, not $Y_3$.  While in $Y_3$ is appeared to be the closure of the {\it negative} full twist (even though it is unknotted in $Y_3$, with the orientation inherited from $Y_1$, it is in fact the closure of the {\it positive} full twist and this is exactly how it embeds in $Y_1$.  In particular, it is the closure of the braid
\[g_1 \sigma_1 g_1^{-1} \cdots g_n \sigma_1 g_n^{-1} = \Delta_d^2.\]
With this orientation-reversal in mind, it is now clear that the $i^{\text{th}}$-component of $L_2$ is actually the boundary of a band determined by $\phi(g_i)(\gamma_1)$ with a positive half-twist.  Thus, the ribbon surface given as the union of $\cD_1$ and $\cD_2$ is the same ribbon surface obtained by Rudolph and thus this torus diagram determines the original symplectic surface, up to smooth isotopy.
\end{proof}

\begin{figure}[h!]
\centering
\labellist
	\large\hair 2pt
	\pinlabel $+$ at 215 303
	\pinlabel $-$ at 132 220
\endlabellist
\includegraphics[width=.5\textwidth]{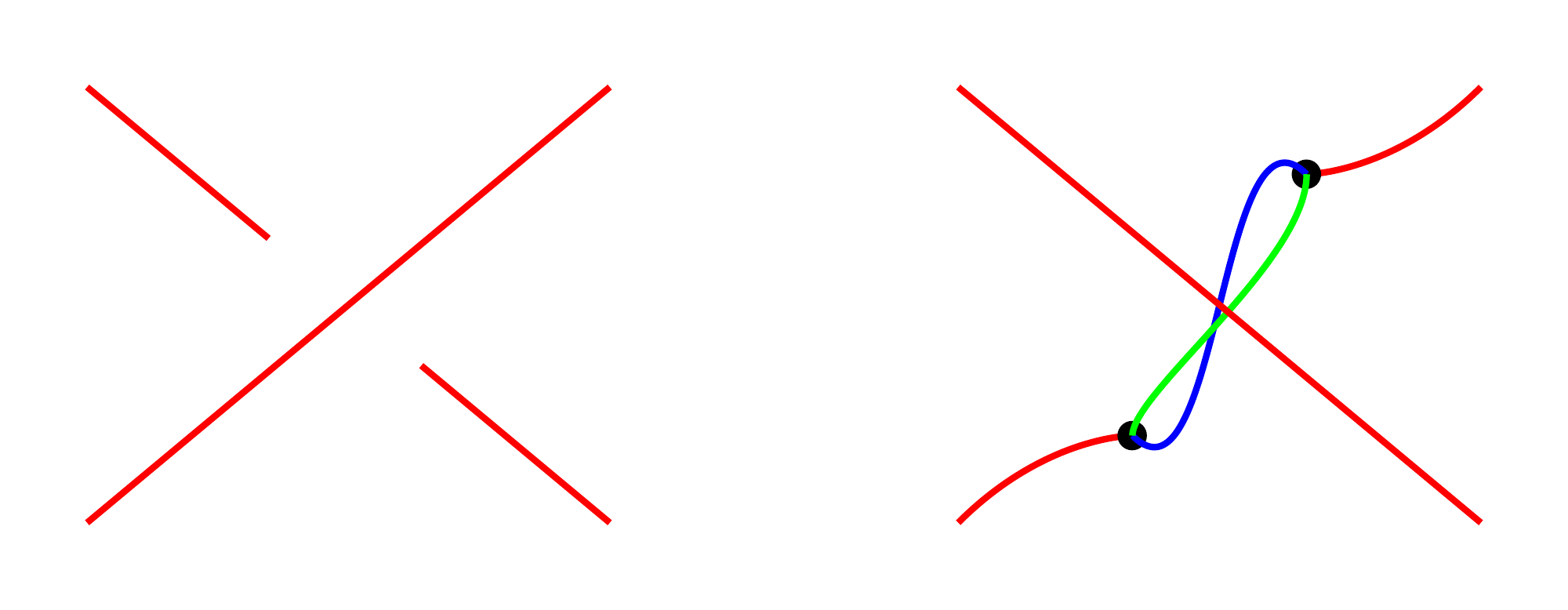}
\caption{A crossing of $\tau_{\alpha}$ ({\it left}) can be removed by a mini stabilization ({\it right}).}
\label{fig:mini-stabilization-crossing}
\end{figure}

\pagebreak

\begin{proposition}
Let $F$ be a braided surface with only $A$-singularities.  Then $F$ can be isotoped into transverse bridge position.
\end{proposition}

\begin{proof}
The proof proceeds exactly as the proof of Proposition \ref{prop:smooth-torus-diagram}, except that for an $A_n$ singularity, we replace the unknot component of $L_2$ in Figure \ref{fig:pos-conj} with a $T(2,n+1)$ component.  This can be achieved while maintaining transverse bridge position.  Figure \ref{fig:cusp-local} depicts this for an $A_2$-singularity and the general case can be achieved similarly.
\end{proof}

\begin{figure}[h!]
\centering
\includegraphics[width=.1\textwidth]{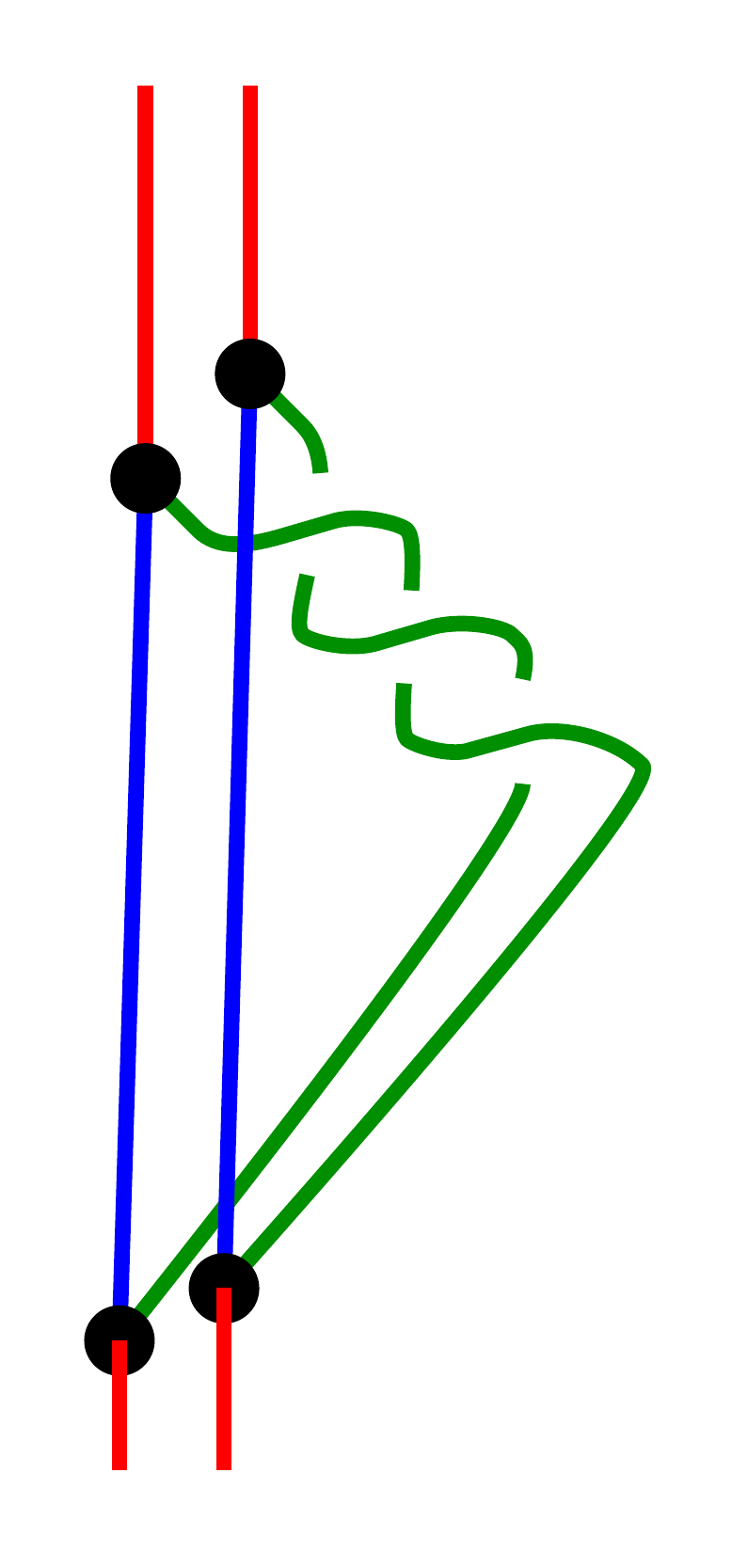}
\caption{To encode an $A_2$-singularity, the unknot component of $L_2$ in Figure \ref{fig:pos-conj} can be replaced by a trefoil component, while maintaining transverse bridge position.}
\label{fig:cusp-local}
\end{figure}

Theorem \ref{thrm:branch-locus} now follows immediately.

\bibliographystyle{alpha}
\nocite{*}
\bibliography{References}


\end{document}